\documentclass[11pt]{article}
\usepackage{amsthm}
\newtheorem{tm}{Theorem}[section]
\newtheorem{lm}[tm]{Lemma}

\newtheorem{cor}[tm]{Corollary}

\newtheorem{??}[tm]{Question}

\usepackage{graphicx}
\usepackage{amssymb}
\usepackage{import}
\usepackage{xifthen}
\usepackage{pdfpages}
\usepackage{transparent}
\usepackage{color}

\font\tenmsb=msbm10
\font\sevenmsb=msbm7
\font\fivemsb=msbm5

\newfam\msbfam
\textfont\msbfam=\tenmsb
\scriptfont\msbfam=\sevenmsb
\scriptscriptfont\msbfam=\fivemsb
\def\Bbb#1{{\fam\msbfam #1}}

\font\teneufm=eufm10
\font\seveneufm=eufm7
\font\fiveeufm=eufm5
\newfam\eufmfam
\textfont\eufmfam=\teneufm
\scriptfont\eufmfam=\seveneufm
\scriptscriptfont\eufmfam=\fiveeufm

\newcommand\comp{{\mathbb{C}}}
\newcommand\real{{\Bbb R}}

\newcommand\Real{{\rm{Re}}}
\newcommand\Imag{{\rm{Im}}}

\title{On Simply Connected Domains Supporting an Unbounded Analytic Function with Bounded Derivative}
\author{Cameron MacMahon}
\date{September 2026}

\begin{document}

\maketitle

\begin{abstract}
In this paper we provide a geometric characterization of simply connected domains in the complex plane supporting an unbounded analytic function with bounded derivative, providing a partial answer to Problem 1.45 in the Hayman List \cite{hayman}. For a given simply connected domain this characterization is achieved by estimating above and below the norms of certain evaluation functionals on a suitable Banach space of analytic functions on that domain by geometric quantities. 
\end{abstract}

\section{Introduction}

Let $\Omega$ be a domain in $\comp$. An analytic function $f$ on $\Omega$ is called $\emph{strongly unbounded}$ if there exists a sequence $\{z_m\}_{m \in \Bbb{N}} \subset \Omega$ along which $f$ and all of its derivatives tend to infinity. Following James Hinchliffe \cite{hinch}, we call a domain $\Omega$ a $\emph{Rubel Domain}$ if every unbounded analytic function on $\Omega$ is strongly unbounded. These domains are named after Lee Rubel, who proposed \cite{rubel} that the unit disc $\Bbb{D}$ had this property. This was confirmed by Alexander Gordon \cite{gordon}, who invoked an inductive argument to construct a sequence along which all derivatives of a given analytic function $f$ tend to infinity from a sequence along which only $f$ is known to tend to infinity. Hinchliffe \cite{hinch} used a variation of this inductive argument to prove that all quasidiscs are Rubel domains. In Walter Hayman's $``$Research Problems in Function Theory," Problem 1.45 due to James Langley \cite{hayman} asks the reader to $``$Determine necessary and/or sufficient conditions for a bounded plane domain to be a Rubel domain." 

\bigskip

\noindent The simplest obstruction to a domain $\Omega$ being Rubel is the existence of an unbounded analytic function whose first derivative is bounded. For instance, take $f(z) = z$ on any unbounded plane domain and observe that any Rubel domain must be bounded. Even bounded domains can possess such a function (see Figure 1). \begin{figure}[ht]
    \centering
    \def\svgwidth{.6\textwidth}
    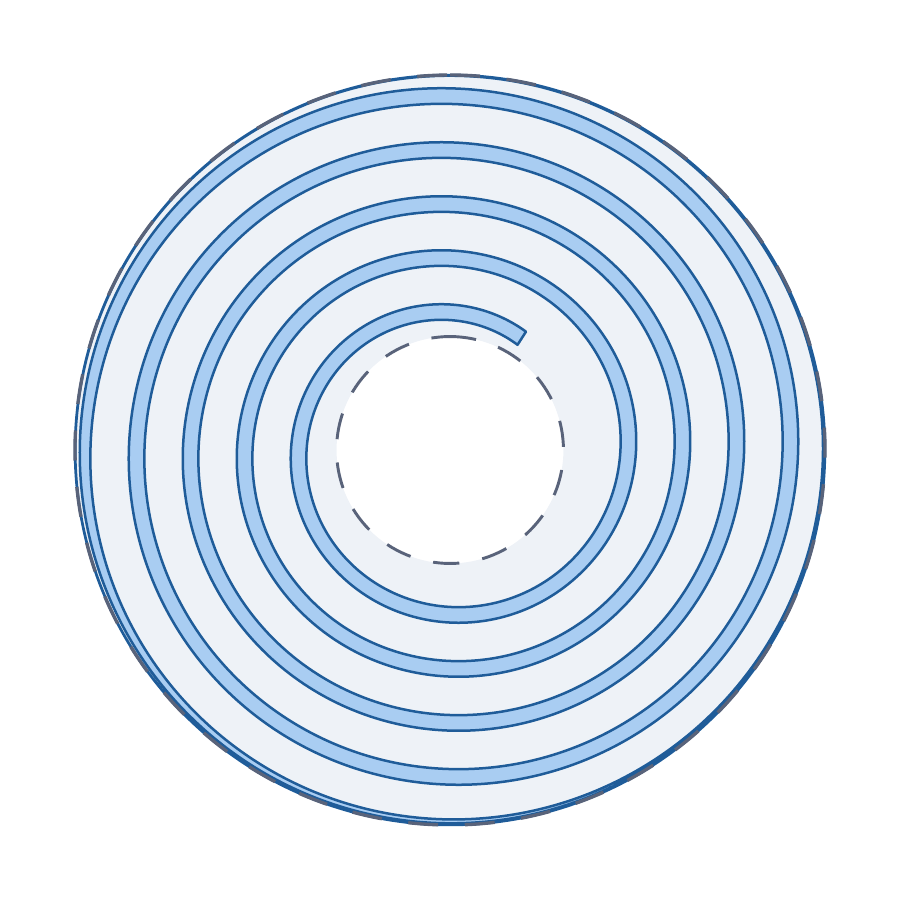

    \caption{A Bounded Spiral Domain in an annulus about $0$ on which $\log(z)$ is unbounded.}
    \label{fig:circles}
\end{figure} As a first step towards Langley's question, it would therefore be desirable to determine a characterization of those planar domains that possess an unbounded analytic function whose first derivative is bounded. The purpose of this paper is to provide such a characterization in the case that $\Omega$ is simply connected. We prove

\begin{tm}
    A simply connected domain $\Omega$ in the complex plane supports an unbounded analytic function whose first derivative is bounded if and only if $\Omega$ has infinite diameter in its interior path metric.  
\end{tm}

\noindent Given a domain $\Omega$, the interior path metric $d_\Omega(p,q)$ is a metric on $\Omega$ defined for two points $p, q \in \Omega$ to be the infimal length among rectifiable curves lying entirely in $\Omega$ joining $p$ to $q$. Thus $\Omega$ has infinite diameter in its interior path metric if and only if for any large $M > 0$ we may find two points inside $\Omega$ so that the infimal length of rectifiable paths joining them and lying entirely inside $\Omega$ is greater than or equal to $M$. Note that, in accordance with Theorem 1.1, the domain in Figure 1 has unbounded diameter in its interior path metric. 

\bigskip

\noindent Theorem 1.1 is proved by comparing $d_\Omega(p,q)$ to a quantity depending only on the analytic functions on $\Omega$ as follows. Let $\mathcal{O}(\Omega)$ denote the set of analytic functions on $\Omega$. Let $X = \{f \in \mathcal{O}(\Omega) : |f'| \leq 1\}$ and let $X(p,q) := \sup_{f \in X} |f(p) - f(q)|$. Theorem 1.1 is a consequence of 

\begin{tm}
   There is a $c > 0$ such that on any simply connected domain $\Omega \subset \comp$,  and for all $p,q \in \Omega$

    $$cd_\Omega(p,q) \leq X(p,q) \leq d_\Omega(p,q)$$
    
\end{tm}

\noindent  Theorem 1.1 provides a geometric characterization of those simply connected domains for which, given any unbounded analytic function, there is always a sequence along which that function and its first derivative tend to infinity. Define the set of domains $Rubel(k)$ to be the set consisting of those domains $\Omega \subset \comp$ for which every unbounded analytic function $f$ on $\Omega$ has a sequence $\{z_n\}_{n \in \Bbb{N}} \subset \Omega$ upon which $f, f', f'' ..., f^{(k)}$ all tend to infinity,  with standard Rubel domains being denoted $Rubel(\infty)$. We reserve the terminology $Rubel_0(k)$ for simply connected sets belonging to $Rubel(k)$. For instance, Theorem 1.1 gives a simple geometric characterization of the set of domains $Rubel_0(1)$ while Problem 1.45 above \cite{hayman} asks for a characterization of $Rubel(\infty)$. 

\bigskip

\noindent Note that $Rubel(k+1) \subset Rubel(k)$ and $Rubel_0(k+1) \subset Rubel_0(k)$.  Looking to extend our result, we might ask if being simply connected and bounded for the interior path metric is enough to guarantee $Rubel_0(\infty)$. In other words, we may ask if all of these inclusions of sets are in fact equalities of sets. This is false. We can construct a simply connected domain $\Omega$ which is bounded for its interior path metric and an unbounded analytic function on it so that there is no sequence along which $f$ and all of its derivatives tend to infinity in modulus along it. In fact, in Section 4 we prove something stronger:

\begin{tm}
    $Rubel_0(1) \not\subset Rubel_0(3)$. In particular, $Rubel(1) \not\subset Rubel(3)$.
\end{tm}

\noindent This theorem will be proved by first constructing a real analytic function on $(0,1)$ with desirable properties and analytically continuing it to a small open neighborhood of $(0,1)$. Theorem 1.3 is the first strict inclusion in the chain of inclusions $Rubel(k + 1) \subset Rubel(k)$. We prove in Section 5: 

\begin{tm}
    $Rubel(1) = Rubel(2)$. In particular, $Rubel_0(1) = Rubel_0(2)$. 
\end{tm}

\noindent This will follow from Theorem 5.1, a very general statement about smooth functions defined on Jordan rays. As an immediate corollary, the geometric condition of bounded diameter for the interior path metric actually extends to characterize $Rubel_0(2)$ domains as well.

\begin{cor}
    $Rubel_0(2)$ consists of precisely those simply connected domains which have bounded diameter for their interior path metric. 
\end{cor}

\noindent This is also proven in Section 5. Section 6 contains a discussion of open questions. 

\bigskip 

\noindent We conclude with a few remarks before we proceed, which the impatient reader may skip. The geometric condition put forth in Theorem 1.1 is weaker than assuming that $\Omega$ supports a hyperbolic geodesic ray of infinite Euclidean length. Indeed, consider a comb domain (See Figure 2) with $N$ slits in the $N$'th tooth. \begin{figure}[ht]
    \centering
    \def\svgwidth{.6\textwidth}
    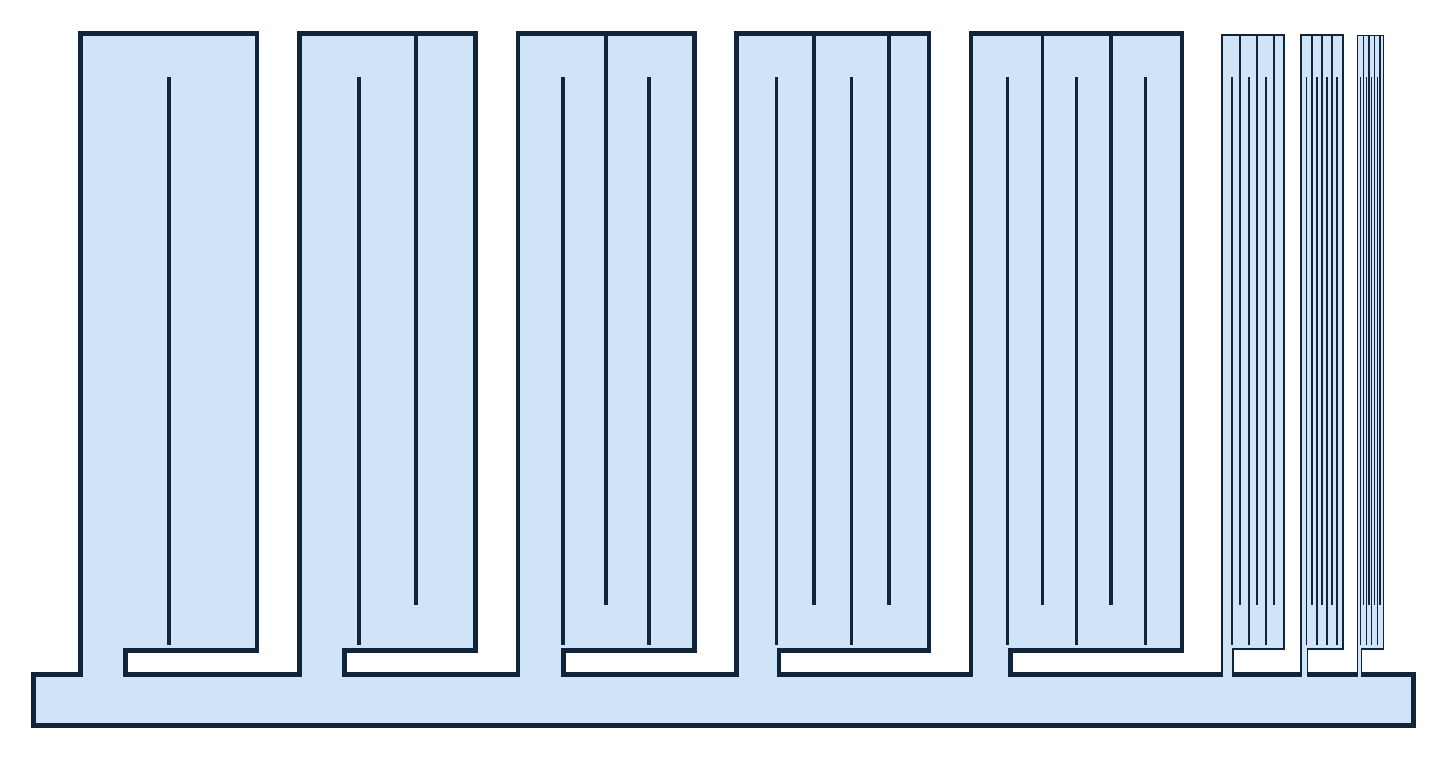

    \caption{A Comb Domain with $N$ slits in the $N$'th tooth}
    \label{fig:circles}
    \end{figure} The Gehring-Hayman Theorem \cite {gehhay} states that the Euclidean length of a hyperbolic geodesic segment (perhaps landing on the boundary) is bounded above by a constant multiple of the Euclidean length of any path joining the two endpoints. We may find such paths with finite length for hyperbolic geodesic rays in this comb domain, so there is no hyperbolic geodesic ray of infinite Euclidean length in this domain. However, there is a sequence of hyperbolic geodesic segments of Euclidean length tending to infinity, as they must each zig-zag through $N$ slits as $N \rightarrow \infty$. Therefore, this domain is unbounded in diameter for its interior path metric.  

    \bigskip

    \noindent Finally, we remark that Hinchliffe proves that $\emph{John domains}$ are Rubel domains as well, a larger well-known class of domains than the set of quasidiscs. In his terminology, Hinchliffe proves that any bounded domain in $\comp$ which contains a $\delta$-visible compact set is a Rubel domain. By the fact that John domains may not contain outward pointing cusps, they contain $\delta$-visible compact sets. For more details, see $\cite{hinch}$.

\subsection{Acknowledgments}

I would like to thank Professor Christopher Bishop without whom this paper would not exist, both for bringing this question to my attention and for discussing it with me at length.

\subsection{LLMs in Function Theory}
Following a suggestion of Professor Kirill Lazebnik, Professor Christopher Bishop, and James Waterman this paper is aimed in part at investigating efficient use of LLMs in function theory. Gemini Pro and Claude Opus 5 contributed to this paper in the following ways. The author suspected initially that the presence of a hyperbolic geodesic ray with infinite Euclidean length might signal the presence of an unbounded analytic function with bounded derivative. Following a suggestion of Christopher Bishop the assumption was revised to infinite diameter in the interior path metric as in Theorem 1.1. It was the author's idea to use the Uniform Boundedness Principle (Theorem 2.1) and the Fundamental Theorem of Calculus to construct the desired function. In discussion around these ideas, Gemini conjectured that $X(p,q) = d_\Omega(p,q)$ for all $p,q \in \Omega$. While the author knew this to be false, this lead the author to conjecture Theorem 1.2 in its current form. Similarly, it was the author's idea to construct a Bloch function whose imaginary part was bounded in modulus along a long curve. Opus 5's primary contribution was generating a sketch of the calculations and argumentation proving Theorem 1.2. The author is responsible for working out this argument completely and in detail. This amounted to correcting errors in computing constants, working out and checking estimates Opus 5 did not prove rigorously, and adding references to theorems (for example work of Mattner \cite{mattner}) where Opus 5 lacked justification or implicitly used a statement. Opus 5 also directed the author's attention to $\cite{whit}$. All of what is written below was written by the author - LLM writing advice was ignored. Future papers in this series will contain similar detailed statements to differentiate them from work done without LLM assistance. 

\section{Proof of Theorem 1.1 from Theorem 1.2}

\noindent We prove the theorems stated in the introduction. We begin by assuming Theorem 1.2 and using it to prove Theorem 1.1. The proof uses in a crucial way Theorem 2.1, the following version of the Banach-Steinhaus Uniform Boundedness Principle \cite{folland}:

\begin{tm}[Banach-Steinhaus, Sequence Version]
Suppose $X$ is a Banach Space and $Y$ is a normed vector space. Suppose $\{T_n\}_{n \in \Bbb{N}}$ is a sequence of bounded operators from $X$ to $Y$ such that the sequence $||T_n||$ is unbounded. Then there exists $x \in X$ such that $||T_n(x)||$ is an unbounded sequence. 
    
\end{tm}

\begin{proof}[Proof of Theorem 1.1]
Assume Theorem 1.2. Let $\Omega$ be a simply connected domain in the complex plane, and suppose $\Omega$ is unbounded in its interior path metric. Let $X = \{f \in \mathcal{O}(\Omega) : |f'| \leq 1\}$. Let $p,q \in \Omega$. Let $X_p := \{f \in \mathcal{O}(\Omega): f(p) = 0, |f'| <  \infty\}$. $X_p$ is a Banach space upon taking the norm to be the $L^\infty$-norm applied to $f'$, $f \in X_p$. More precisely, $||f||_{X_p} := ||f'||_\infty$. By the fact that any function in $X$ may be translated without changing $|f(p) - f(q)|$, without loss of generality it is true that  $X(p, q) = \sup_{f \in X_p, ||f||_{X_p} \leq 1} |f(q)|$.  Define the evaluation operator $\Lambda_q (f) := f(q)$. We observe that the operator norm $||\Lambda_q||_{X_p^*} = X(p,q)$ is finite, and so $\Lambda_q \in X_p^*$. Since $\Omega$ is unbounded in its interior path metric, choose a sequence of $q_n \in \Omega$ so that $d_\Omega(p, q_n) \rightarrow \infty$ as $n \rightarrow \infty$. By Theorem 1.2, we have that $cd_\Omega(p,q) \leq X(p,q) \leq d_\Omega(p,q)$ for all $q \in \Omega$, and so $X(p, q_n) \rightarrow \infty$ as $n \rightarrow \infty$. This implies that $||\Lambda_{q_n}||_{X_p^*} \rightarrow \infty$ as $n \rightarrow \infty$. Theorem 2.1 now asserts that there is $f \in X_p$ so that $|\Lambda_{q_n}(f)| \rightarrow \infty$ as $n \rightarrow \infty$, but this just means that $f$ grows to infinity along $q_n$ and is therefore unbounded. However, $f$ belongs to $X_p$, and so has bounded derivative. Therefore $\Omega$ supports an unbounded analytic function with bounded derivative.

\bigskip

\noindent Suppose now, still assuming Theorem 1.2, that $\Omega$ supports an unbounded analytic function with bounded derivative. Call this function $f$. We may assume $f(p) = 0$ for some $p \in \Omega$ without loss of generality. Further assume that $f$ has been scaled so that $f$ is in the unit ball of $X_p$. Let $q_n \in \Omega$ be such that $f(q_n) \rightarrow \infty$ as $n \rightarrow \infty$. By the above, $X(p, q_n) \rightarrow \infty$ since $|f(q_n)| \leq  ||\Lambda_{q_n}||_{X_p^*} = X(p, q_n)$. By Theorem 1.2, $d_\Omega(p, q_n) \rightarrow \infty$ as $n \rightarrow \infty$, and so $\Omega$ is unbounded in diameter for its interior path metric. 

\end{proof}

\section{Proof of Theorem 1.2}

\noindent We prove that there is a $c > 0$ such that on any simply connected domain $\Omega \subset \comp$,  and for all $p,q \in \Omega$

    $$cd_\Omega(p,q) \leq X(p,q) \leq d_\Omega(p,q)$$
    
    \noindent The proof is broken into a number of pieces for easier reading. To begin, we prove 

\begin{lm}
    $$X(p,q) \leq d_\Omega(p,q)$$
\end{lm}

\begin{proof}
    The proof is as follows. Let $X = \{f \in \mathcal{O}(\Omega) : |f'| \leq 1\}$, and let $f \in X$. Let $\gamma$ be any rectifiable arc in $\Omega$ joining $p$ to $q$. By the Fundamental Theorem of Calculus, we have

    $$|f(p) - f(q)| = |\int_\gamma f'(w) dw| \leq \int_\gamma |f'(w)| |dw| \leq \int |dw| = \ell(\gamma)$$

    \noindent Taking an infimum over all such $\gamma$, we get that $|f(p) - f(q)| \leq d_\Omega(p,q)$. Since this is true for all $f \in X$, taking a supremum yields Lemma 3.1.

\end{proof}

    \begin{lm} Let $p, q \in \Omega$, and let $S = \{w \in \comp : |\Imag(w)| < 1\}$. Let $\psi: S \rightarrow \Omega$ be a Riemann mapping normalized so that $\psi(a) = p, \psi(b) = q$, $a, b \in \Bbb{R}$, and $a < b$. Then
         $$\int_a^b |\psi'| \geq d_\Omega(p, q)$$
    \end{lm}

    \begin{proof}To prove this, we have that $\psi([a, b]) := \gamma$ a path in $\Omega$ connecting $p$ to $q$. The integral on the left hand side represents the length of $\gamma$. Since $d_\Omega(p,q)$ is the infimal length of all such paths, it is less than the length of $\gamma$ and the proof is complete. 
    \end{proof}

    \bigskip

    \noindent We now proceed to arrive at some estimates on the derivatives of $\psi$. Since $\psi$ is univalent, we let $g = \log(\psi') = A + iB$. We now apply some standard distortion estimates to $\psi$ (these amount to the fact that $g$ is Bloch):
    
    \begin{lm}
        The harmonic function $B(z)$ considered as a smooth function $B(x)$ on $\Bbb{R}$ satisfies $|\partial_xB(x)| \leq 4$, $|\partial^2_xB(x)| \leq 16$ on $\Bbb{R}$, and $|B(x)| \leq |B(0)| + 4|x|$ for all $x \in \Bbb{R}$.
    \end{lm}

        \begin{proof} Let $x \in \Bbb{R}$. $\psi$ is univalent on $D(x, 1) \subset S$ (discs are considered open unless otherwise specified). Therefore, we may consider $\psi |_D$ as a univalent function on a disc centered at a fixed $x \in \Bbb{R}$ with power series (we have translated the function so $\psi(x) = 0$, 
    
    $$\frac{\psi|_D(z)}{\psi|_D'(x)} = (z - x) + a_2(z - x)^2 + ...$$

    \noindent we therefore have

    $$|\frac{\psi|_D''(x)}{\psi|_D'(x)}| = 2|a_2| \leq 4$$

    \noindent by the Koebe Theorem. Note that this implies $|g'(x)| \leq 4$, since $g'(x) = \psi|_D''(x)/\psi|_D'(x) = \psi''(x) / \psi'(x)$ there. Furthermore, for $w \in S$ with $|\Imag(w)| \leq \frac{1}{2}$ we have that $D(w, \frac{1}{2}) \subset S$ as well, so the rescaled Koebe Theorem implies that $|g'(w)| \leq \frac{4}{1/2} = 8$ uniformly in this smaller strip. In particular, for $x \in \Bbb{R}$ we may now apply for the first derivative of $g'$ (the second derivative of $g$) the classical Cauchy derivative estimate for the circle $|\zeta - x| \leq \frac{1}{2}$ to obtain

    $$|g''(x)| \leq \frac{8}{1/2} = 16$$

    \noindent Since $g$, $g'$ and $g''$ are analytic, the distortion on tangent spaces given by $g$ and $g'$ respectively is uniform in all directions. Therefore, considering $g(x)$ as a complex-valued function on $\Bbb{R}$, we have $|\partial_x g(x)| \leq 4$ and $|\partial^2_x g(x)| \leq 16$ on $\Bbb{R}$. By the triangle inequality we therefore have that $|\partial_xB(x)| \leq 4$ and $|\partial^2_xB(x)| \leq 16$ on $\Bbb{R}$. $B$ is harmonic on $S$ and so its restriction to $\Bbb{R}$ is smooth. By the triangle inequality,the Fundamental Theorem of Calculus,

$$|B(x)| = |B(x) - B(0) + B(0)| \leq |\int_0^x (\partial_x B)(a) da| + |B(0)|$$

\noindent Thus, the partial derivative estimate above yields

$$|B(x)|  \leq 4|\int_0^x  da| + |B(0)| = |B(0)| + 4|x|,$$
 
\noindent  completing the proof of Lemma 3.3. 
\end{proof}

    \bigskip

  \begin{proof} We now use these lemmas to complete the proof of Theorem 1.1. Lemma 3.3 and convolution with an approximate identity are used to $``$straighten" the argument of $\psi$ on $S$ in order to eliminate cancellation of the argument that may appear in taking a primitive for $\psi^{-1}$ on $\Omega$. In particular, we will approximate $B(x)$ along the real axis by the restriction of an analytic function. Below, $B'(x)$ denotes $\partial_x B(x)$. With $\sigma = 1/3$, consider the entire function

    $$k(w) = \frac{1}{\sigma \sqrt{2\pi}}\exp({-\frac{w^2}{2\sigma^2}})$$

    \noindent Furthermore, define (formally for now) 

    $$G(w) := \int_{-\infty}^\infty B(x)k(w - x)dx$$

    \noindent By the fact that $B(x)$ grows at most linearly (Lemma 3.3), $|B(x)| \leq |B(0)| + 4|x|$, and $k$ decays exponentially a textbook combination of Fubini's Theorem and Morera's Theorem (for a very general statement see \cite{mattner}) proves that $G(w)$ indeed defines an entire function which may be differentiated under the integral sign. Moreover, $G$ is real on $\Bbb{R}$ and well-approximates $B(x)$ there. Precisely, we have for $w = a + ib$, $x = a - s$: 

    $$|G(a) - B(a)| = |\int_{-\infty}^\infty B(a - s)k(s)ds - \int_{-\infty}^\infty (B(a) + sB'(a)) k(s)ds |$$

    \noindent since $k(s)$ integrates to $1$ over the real axis and $sk(s)$ integrates to zero over the real axis. This is because $sk(s)$ is the product of an even function $k(s)$ and an odd function $s$ so is odd. Therefore

    $$|G(a) - B(a)| = |\int_{-\infty}^\infty (B(a - s) - B(a) - sB'(a))k(s) ds|$$

    \noindent We now use Taylor's Theorem to see that

    $$|B(a -s) - B(a) - sB'(a)| \leq \frac{s^2}{2}|B''(a)|$$

    \noindent In particular, Lemma 3.3 tells us that $|B''(a)| \leq 16$, so have that 

    $$|G(a) - B(a)| \leq  \int_{-\infty}^\infty |(B(a - s) - B(a) - sB'(a))k(s)|ds \leq \int_{-\infty}^\infty 8 s^2 k(s)ds = 4\sigma^2 < \pi/3$$

    \noindent So, in particular, $G(a)$ is real on the real axis and approximates the argument of $\psi'$ there.  Furthermore, since
    
    $$\int_{-\infty}^\infty \Imag(k(a + ib))da = 0$$

    \noindent for all $b$, and Lemma 3.3 tells that  $|B(a - s) - B(a)| \leq 4|s|$, we have for all $w = a+bi \in S$:

    $$|\Imag(G(w))| = |\int_{-\infty}^\infty (B(a - s) - B(a)) \Imag(k(s + ib))ds|$$

    \noindent which is less than or equal to

     $$\int_{-\infty}^\infty |(B(a - s) - B(a))| |\Imag(k(s + ib))|ds$$

    \noindent which in turn is less than or equal to 

    $$\int_{-\infty}^\infty 4|s| |\Imag(k(s + ib))|ds$$

    \noindent We now estimate from above $|\Imag(k(s + ib))|$ for $s +ib \in S$:

    $$|\Imag (k(s + ib))| = |-\frac{1}{\sigma \sqrt{2\pi}} \exp(\frac{b^2 - s^2}{2\sigma^2}) \exp(\frac{-isb}{\sigma^2})| \leq |\frac{1}{\sigma\sqrt{2\pi}} \exp(\frac{1}{2\sigma^2}) \exp({\frac{-s^2}{2\sigma^2}})| $$

    \noindent In particular, the exponential decay absorbs the linear growth coming from the $B$ terms as follows:

    $$\int_{-\infty}^\infty 4|s| |\Imag(k(s + ib))|ds \leq \int_{-\infty}^\infty4|s| |\frac{1}{\sigma\sqrt{2\pi}} \exp(\frac{1}{2\sigma^2}) \exp({\frac{-s^2}{2\sigma^2}})| $$

    \noindent which is equal to

    $$\frac{4}{\sigma\sqrt{2\pi}} \exp({\frac{1}{2\sigma^2}}) \int_{-\infty}^\infty |s| \exp({\frac{-s^2}{2\sigma^2}}).$$

    \noindent Next, we realize that the integrand is even, and so 

    $$\int_{-\infty}^\infty |s| \exp({\frac{-s^2}{2\sigma^2}}) = 2\int_0^\infty s \exp({\frac{-s^2}{2\sigma^2}}) = 2\sigma^2$$

    \noindent by explicit anti-differentiation as $\frac{d}{ds}(-\sigma^2 \exp({\frac{-s^2}{2\sigma^2}})) = s\exp({\frac{-s^2}{2\sigma^2}}).$ Therefore we have

    $$|\Imag(G(w))| \leq \lambda := \frac{4}{\sigma\sqrt{2\pi}} (2\sigma^2)  (e^{\frac{1}{2\sigma^2}}) \approx 95.76$$

    \noindent for an absolute $\lambda> 0$ depending only on $\sigma$. We may now construct $f$ analytic on $\Omega$ for which $|f(p) - f(q)| \geq \frac{1}{2} e^{-\lambda}d_\Omega(p,q)$ completing the proof of the lower bound with constant $c = \frac{1}{2}e^{-\lambda} \approx 1.29 \times 10^{-42}$. To do so, on $S$ let $h(z) := \psi'(z) e^{-\lambda - iG(z)}$. We have that $h(z)$ is analytic because linear combinations and compositions of analytic functions are analytic. Let

    $$F(w) := \int_a^w h(\zeta) d\zeta$$

    \noindent and let

    $$f :=F \circ \psi^{-1}$$

    \noindent on $\Omega$. We have that $|h| = |\psi'|e^{-\lambda + \Imag(G)} \leq |\psi'|$ on $S$ because $|\Imag(G)| \leq \lambda$ there. This yields a derivative bound on $f$, using the standard expression for the derivative of the inverse of a univalent map as the reciprocal of the derivative of the original map:

    $$f' = \frac{h}{\psi'} \circ \psi^{-1}$$  
    
    \noindent and so
    
    $$|f'| \leq 1$$

    \noindent In particular, $f \in X$. Given $x \in \Bbb{R}$, $G(x)$ is real, and furthermore $\psi' = |\psi'|e^{iB}$, so

    $$h(x) = |\psi'|e^{-\lambda}e^{i(B(x) - G(x))}$$

    \noindent Now recall that $|B(x) - G(x)| \leq \frac{\pi}{3}$. Since the argument is constrained to this cone, we have $\Real(h(x)) \geq \frac{1}{2}|h(x)| = \frac{1}{2}e^{-\lambda}|\psi'(x)| \geq 0$ on $\Bbb{R}$. Since $\psi(a) = p$ and $\psi(b) = q$

    $$|f(p) - f(q)| = |F(a) - F(b)| = |\int_a^b h(z) dz| \geq |\Real\int_a^bh(z) dz| = |\int_a^b \Real(h(z))dz| $$

    \noindent Now, $\Real(h(x)) \geq 0$ pointwise on $[a,b]$. Therefore, we have

    $$|f(p) - f(q)| \geq  \int_a^b \Real(h(x))dx \geq \frac{1}{2}e^{-\lambda} \int_a^b |\psi'(x)|dx \geq \frac{1}{2}e^{-\lambda}d_\Omega(p,q)$$

    \noindent because of Lemma 3.2. Theorem 1.1 is therefore proved with $c = \frac{1}{2}e^{-\lambda} \approx 1.29 \times 10^{-42}$. 

    \end{proof}

\section{Proof of Theorem 1.3}

\noindent We construct an example of a simply connected domain in $Rubel_0(1)$ which is not in $Rubel_0(3)$. In particular, $Rubel(1) \not\subset Rubel(3)$.

\begin{proof}
   We have demonstrated above that $Rubel_0(1)$ is precisely the set of simply connected domains which are bounded in diameter for their interior path metric. Thus, it suffices to find a simply connected domain $\Omega$ with bounded diameter in its interior path metric and an unbounded analytic function $f$ on $\Omega$ so that along no sequence do all of $f, f', f'', f'''$ tend to infinity  in modulus simultaneously. 

    \bigskip

    \noindent It suffices to construct an unbounded smooth function $f: (0,1) \rightarrow \Bbb{R}$ with $\min(|f|, |\partial_x^{(3)}f|)$ uniformly bounded on $(0,1)$ and to approximate that function in the $C^3$ topology by a real analytic function which continues analytically to a simply connected neighborhood of $(0,1)$ which is bounded in diameter for its interior path metric. Because the resulting analytic map and all of its derivatives have conformal linear derivative on each tangent space, $|\partial_x^{(3)}f| = |f'''|$, and so this construction will yield an unbounded analytic function $f$ on a simply connected neighborhood of $(0,1)$ which is bounded in diameter for its interior path metric such that $\min(|f|, |f'''|)$ is uniformly bounded. The last statement is because the neighborhood onto which we analytically continue can be chosen as small as we wish, and these quantities vary continuously. 

    \begin{figure}[ht]
    \centering
    \def\svgwidth{.6\textwidth}
    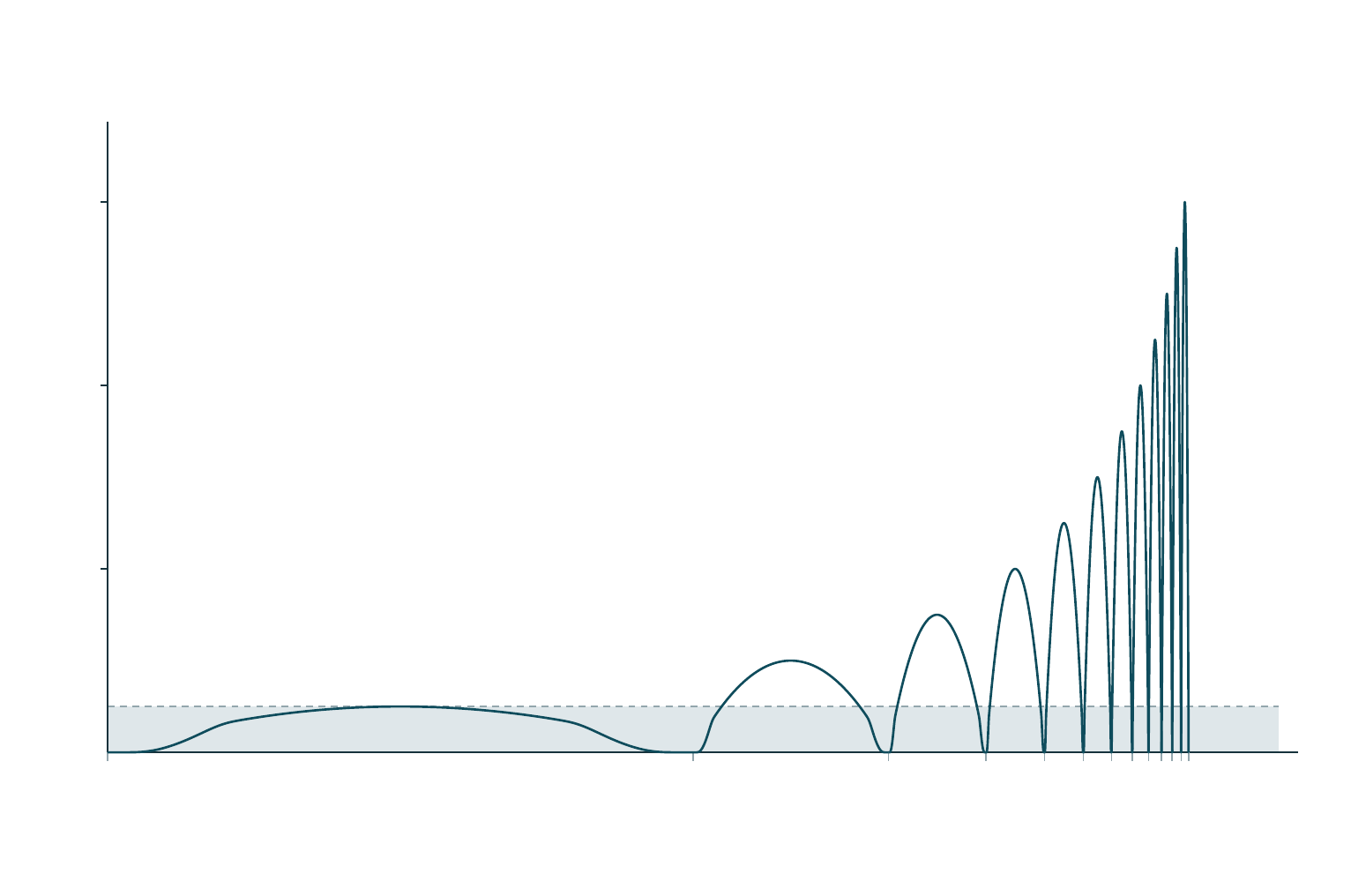

    \caption{The Function $f(x)$}
    \label{fig:circles}
\end{figure}

    \noindent Tile $[0,1)$ by $I_j = [t_j, t_{j+1}] := [1 - \frac{1}{j}, 1 - \frac{1}{j+1}]$, $j \geq 1$. Let $u_j$ be the unique parabola on $I_j$ vanishing at the endpoints with maximum value $1$ on $u_j$ so that $\partial_x^{(3)}u_j \equiv 0$ on the interior of $I_j$, i.e.
    
    $$u_j(x) = c_j(x - t_j)(t_{j+1} - x)$$
    
   \noindent  Let $\chi \in C^\infty(\Bbb{R})$ be a smooth function identically zero on $(-\infty, 0]$, identically $1$ on $[1, \infty)$, $0 \leq \chi \leq 1$ on $[0,1]$. Let $\delta_j > 0$ be small, and let 

   $$f(x) = ju_j(x)\chi\big(\frac{x - t_j}{\delta_j}\big)\chi\big(\frac{t_{j+1} - x}{\delta_j}\big)$$

   \noindent for $x \in I_j$. This function is smooth on the interior of each $I_j$ since $f(x)$ is a product of three smooth functions and a constant there. Moreover, $f$ may be extended to a smooth function on all of $(0,1)$ such that $f$ is identically $0$ in a neighborhood of each endpoint $t_j$. This is because on $I_j$, $f$ decays to zero before reaching either endpoint because of the two $\chi$ factors (see Figure 3). $f$ is clearly unbounded, as the maximum of $f(x)$ on $I_j$ is $j$. Moreover on $[t_j + \delta_j, t_{j+1} - \delta_j]$, we have by definition of $\chi$ that $|\partial_x^{(3)}f(x)| \equiv 0$ as $f(x) = ju_j(x)$ there, a constant times a quadratic. On a collar such as $[t_j, t_j + \delta_j]$, we observe that by the behavior of $\chi$ in $[0,1]$ and the fact that $u_j$ decays to zero monotonically  on $I_j$ as it approaches an endpoint,  we could choose $\delta_j$ small enough so that $|f(x)| \leq 1$ there. Hence, on $(0,1)$ we have $\min(|f(a)|, |\partial_x^{(3)}f(a)|) \leq 1$. 


    \bigskip

   \noindent We now leverage an approximation argument due to Whitney. Lemma 6 in \cite{whit} states that real analytic functions on $(0,1)$ can approximate a smooth function in the following sense: given some $\epsilon > 0$ and $f$ smooth on $(0,1)$, there exists some real analytic $g_\epsilon \in \omega((0,1))$ so that 

   $$|\partial_x^{(k)}g_\epsilon(x) - \partial_x^{(k)}f(x)| < \epsilon$$

   \noindent for all $0 \leq k \leq 3$. In particular, given $f$ as above, we may find $g$ real analytic on $(0,1)$ such that $\min(|g|, |\partial_x^{(3)}(g)|) \leq C_\epsilon$ for some constant $C_\epsilon$ depending only on some prior chosen $\epsilon$ very small and arbitrary. Given that $g$ is real analytic, it may be  analytically continued into a small open neighborhood of $(0,1)$. Without loss of generality we may choose this domain to be a cusp domain (see Figure 4) of the form $\{x + iy : x \in (0,1), |y| \leq \rho(x), \lim_{x \rightarrow 0}\rho(x) = 0, \lim_{x \rightarrow 1} \rho(x) = 0 \}$, for some $\rho(x)$ continuous on $(0,1)$. Without further loss of generality $\rho(x)$ may be chosen to be uniformly close to zero in order to guarantee that $|g(x + iy) - g(x)| \leq \delta$ and $|\partial_x^{(3)}(x + iy) - \partial_x^{(3)}g(x)| \leq \delta$ for some $\delta$ small. By our remark above, $g$ has conformal linear derivative in each tangent space and we are done.

\end{proof}

   \begin{figure}[ht]
    \centering
    \def\svgwidth{.6\textwidth}
    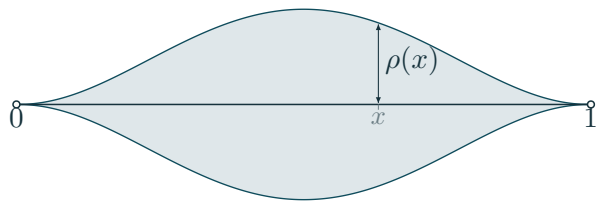

    \caption{A Cusp Domain}
    \label{fig:circles}
\end{figure}

\noindent Some remarks are in order. First note that the cusp shape of the domain is necessary for the argument. If $g$ could be continued to some open neighborhood of $[0,1]$, $g$ could be restricted to be analytic on a quasidisc. Hinchliffe \cite{hinch} proved that all quasidiscs have type $Rubel_0(\infty)$. Therefore, permitted cusps are not restrictions of neighborhoods of $[0,1]$ nor are they asymptotic to cones. It is an interesting question to determine how regular we may permit the cusps to be. May they, for instance, be asymptotic to the cusp formed by two tangent circles? Second, it is likely that the analytic approximant above can be written out explicitly, without appeal to Whitney's Theorem. This may yield more information on the geometry of the necessary cusps. See Section 6.  

\section{Proof of Theorem 1.4}

 We prove that $Rubel(1) = Rubel(2)$. In particular, $Rubel_0(1) = Rubel_0(2)$. We begin by proving the following theorem:

\begin{tm}
    Let $\Gamma \subset \comp$ be a smooth oriented Jordan ray with one endpoint, call it $o \in \comp$. Let $\{x_n\}_{n \in \Bbb{N}} \in \Gamma$ be a sequence of points so that $x_n, x_{n+1}$ are arc-length distance at most $\Delta$ apart along $\Gamma$, $x_0 = o$. Let $M  > 1$, and suppose $f$ is a smooth real valued function on $\Gamma$ such that $f(x_n) \geq M^n$, $f(o) = M$. There exists a sequence $\{y_n\}_{n \in \Bbb{N}} \subset \Gamma$ such that $\min(|f(y_n)|, |f'(y_n)|, |f''(y_n)|) \rightarrow \infty$ as $n \rightarrow \infty$. 
\end{tm}

\begin{proof}
 Passing to a subsequence, assume that the $x_i$ are ordered along $\Gamma$. We first find a sequence $\{z_n\}_{n \in \Bbb{N}} \subset \Gamma$ along which $f(z_n), f'(z_n) \rightarrow \infty$. Parameterizing $\Gamma$ by arc-length, let $t_n$ be the first time $f(t_n) = M^n$. The time $t_n$ must exist by the Intermediate Value Theorem, as $f(x_n) \geq M^n$. Moreover, for each $i \in\{1,...,n\}$ we have $t_i \in [x_1, x_n]$, an interval of length at most $n \Delta$. In particular, the distance between $t_{n-1}$ and $t_n$ is at most $n\Delta$. Let $i \geq 1$. Let $\delta_i \in [t_{i-1}, t_i]$ be the last time $f$ attains the value $M^{i - 1}$.  By the Mean Value Theorem, there exists $z_i \in [\delta_i, t_i] \subset [t_{i - 1}, t_i]$ such that 
 
 $$f'(z_i) = \frac{f(t_i) - f(\delta_i)}{t_{i} - t_{i - 1}} \geq \frac{M^i - M^{i - 1}}{i\Delta} = \frac{M^{i - 1} (M - 1)}{i\Delta} > 0$$

 \noindent  Moreover, since $z_i \in [\delta_i, t_i]$ and $\delta_i$ is the last time in $[t_{i-1}, t_i]$ that $f$ attains the value $M^{i - 1}$ we have that $f(z_i) \geq M^{i - 1}$ as $f$ must oscillate between $M^{i -1}$ and $M^i$ in $[\delta_i, t_i]$. We observe that $z_i$ and $z_{i - 1}$ are separated by an arc-length distance at most $2i\Delta$. 

 \bigskip

 \noindent By the Intermediate Value Theorem, there exists a first time $T_i \in [\delta_i, z_i]$ where $f(T_i) \geq M^{i - 1}$ and $f'(T_i) = \frac{M^{i - 1}(M-1)}{2i\Delta}$. We observe that the arc-length distance between $T_{i - 1}$ and $T_i$ is at most $4i\Delta$. Between $T_{i - 1}$ and $T_i$, since $f'$ oscillates between $f'(T_{i - 1}$ and $f'(T_i)$ and eventually attains $f'(T_i)$ at $T_i$, given $i$ sufficiently large there must be a first time $q_{i - 1}$ where $f'$ exits the strip

 $$|y - f'(T_{i - 1})| < \frac{f'(T_{i - 1})}{2}.$$

 \noindent Moreover, the distance between $q_{i - 1}$ and $T_{i - 1}$ is at most $4i\Delta$. By the Mean Value Theorem, there exists $y_i \in [T_{i - 1}, T_i]$ such that 

 $$|f''(y_i)| = \frac{|f'(q_{i - 1}) - f'(T_{i - 1})|}{|q_{i-1} - T_{i - 1}| } \geq \frac{f'(T_{i - 1})}{8i\Delta}.$$

 \noindent Since $f(y_i)$ is in the strip  $|f'(y_i) - f'(T_{i - 1})| < \frac{f'(T_{i - 1})}{2}$, we have that $f'(y_i) \geq \frac{f'(T_{i-1})}{2}$. Furthermore, $f'$ is positive and bounded away from zero on the strip, so $f$ is increasing on the strip and so $f(y_i) \geq f(T_{i - 1})$. Asymptotically $f$ and $f'$ grow exponentially fast and so the sequence $\{y_n\}_{n \in \Bbb{N}}$ has the desired properties.

\end{proof}

\noindent It is not possible to remove the absolute value sign from the second derivative in this argument. This is due to the fact that $f'$ may escape the strip described either by exiting above or by exiting below, and both can occur. We now prove Theorem 1.4 from Theorem 5.1, demonstrating that it was necessary to pass to the third derivative in Section 4:

\begin{proof}
    It suffices to prove that $Rubel(1) \subset Rubel(2)$. Let $\Omega$ be a domain in $Rubel(1)$, and let $f$ be an unbounded analytic function on it. Let $M > 1$, and let $\{x_n\}_{n \in \Bbb{N}} \subset \Omega$ be a sequence of points along which $|f|$ tends to infinity. Along such a sequence, either the real or imaginary part of $f$ must blow up. Let $u$ denote this real or imaginary part which tends to infinity. Without loss of generality $\{x_n\}_{n \in \Bbb{N}}$ may have been chosen such that $u(x_n) \geq M^n$ and $u(x_0) = M$. We remind the reader that $M$ was chosen arbitrarily, so $f$ unbounded implies that these conditions may be satisfied. Since $\Omega$ has bounded diameter for its interior path metric, the Fundamental Theorem of Calculus $``$easy direction" in the proof of Theorem 1.1 applies. This guarantees  $\Omega$ has finite diameter in its interior path metric even in the case $\Omega$ was multiply connected. Choose an upper bound $\Delta > 0$. By this boundedness property, we may join $x_n$ to $x_{n + 1}$ by a smooth path $\gamma_n \subset \Omega$ of arc-length less than or equal to $\Delta$. Let the path $\Gamma := \cup_{n \geq 0} \gamma_n$. Smooth all of the corners at a very small scale to obtain a smooth path $\Gamma'$. $u|_{\Gamma'}$ is a smooth function on $\Gamma'$. Notice now that $u|_{\Gamma'}$ and  $\{x_n\}_{n \in \Bbb{N}} \subset \Gamma'$ satisfy the hypotheses of Theorem 5.1, and so we may find a sequence $y_n \in \Gamma' \subset \Omega$ for which $\min(|u(y_n)|, |u'(y_n)|, |u''(y_n)|) \rightarrow \infty$ as $n \rightarrow \infty$. By the triangle inequality and the fact that the derivative of an analytic function uniformly expands or contracts a given tangent space in all directions $\min(|f(y_n)|, |f'(y_n)|, |f''(y_n)|) \rightarrow \infty$ as $n \rightarrow \infty$. This implies that $\Omega$ is in the set $Rubel(2)$ by definition.
\end{proof}

\noindent By the proof of Theorem 5.1, Theorem 1.4 is sharp as far as characterizing $Rubel_0(k)$ domains by their interior path diameter is concerned. 

\section{Open Questions}

\noindent We conclude with a list of open questions. Langley's question and its generalization remain open.

\begin{??}
Verify that $Rubel(3) = Rubel(4) = \cdots = Rubel(\infty)$ or find a further distinction. Do any of $Rubel(k)$ admit a clean geometric characterization, $k \in [3, \infty] \cap \Bbb{N}$?
\end{??}

\noindent Let $S \subset \Bbb{N}$. Let $Rubel(S)$ be those domains $\Omega \subset \comp$ such that for any unbounded analytic function $f \in \mathcal{O}(\Omega)$, there exists a sequence along which all of $f^{(k)}$ tend to infinity, $k \in S$.

\begin{??}
Is there a clean geometric description of $Rubel(S)$?
\end{??}

\noindent By Theorem 1.4 any simply connected domain which has bounded diameter for its interior path metric is $Rubel(\{k, k+1, k+2\})$ for any $k \in \Bbb{N}$, for if there were an unbounded analytic function with $f^{(k)}$ unbounded, there must be a sequence along which $f^{(k)}, f^{(k+1)}, f^{(k+2)}$ are all unbounded. Is this all that can be guaranteed?

\begin{??}
    Is there a $Rubel(2)$ domain $\Omega$ and an unbounded analytic function $f \in \mathcal{O}(\Omega)$ such that no more than three consecutive derivatives tend to infinity along any given sequence?
\end{??}

\noindent As was mentioned in Section 4, perhaps the explicit construction of an approximant to the function $f(x)$ in Figure 3 would yield more information on the geometry of the cusp domain onto which it can be continued. However, Charles Fefferman \cite{feff} \cite{feff2} has discovered quantitative improvements of Whitney's Theorem \cite{whit} which in some sense give $``$optimal approximants."

\begin{??}
    Can Fefferman's Theorems give quantitative control over the cusps in the cusp domain in Theorem 1.3? Can the cusps be asymptotic to two tangent circles?
\end{??}

\noindent Finally, a version of the above discussion holds in higher dimensions, prompting the following:

\begin{??}
    Characterize those domains $\Lambda \subset \Bbb{R}^n$ for which, given any unbounded harmonic function $u$ on $\Lambda$, there exists a sequence $\{x_n\}_{n \in \Bbb{N}} \subset \Lambda$ along which $u(x_j)$ $\max_{1 \leq i \leq n} |\partial_i u(x_j)|$, $\max_{1 \leq i \leq n} |\partial^2_i u(x_j)|$, $\cdots$,  $\max_{1 \leq i \leq n} |\partial^{(k)}_i u(x_j)| \rightarrow \infty$ as $j \rightarrow \infty$.
\end{??}

\noindent Different generalizations exist, for instance with the averages of the derivatives of each order instead of a maximum. The ambiguity results from the fact that the derivative is no longer uniformly expansive in all directions. We observe that Theorem 5.1 does not refer to $n=2$, and so may be applied in this context.





\end{document}